\documentclass[12pt]{amsart}
\usepackage{amsfonts,amssymb,latexsym,amsmath, amsxtra,mathtools}
\usepackage[all]{xy}
\usepackage{graphicx}
\usepackage{color}
\newtheorem*{defn}{Definition}
\theoremstyle{plain}
\newtheorem{theorem}{Theorem}[section]
\newtheorem{corollary}[theorem]{Corollary}

\newtheorem{lemma}[theorem]{Lemma}
\newtheorem{proposition}[theorem]{Proposition}
\newtheorem*{conjecture*}{Conjecture}
\newtheorem*{challenge*}{Open Problem}

\theoremstyle{definition}

\theoremstyle{remark}
\newtheorem*{questions}{Questions}
\newtheorem*{remark}{Remark}

\numberwithin{equation}{section}

\newcommand{\R}{\mathbb R}
\newcommand{\N}{\mathbb N}
\newcommand{\Z}{\mathbb Z}
\newcommand{\C}{\mathbb C}
\def\H{\mathbb H}

\newcommand{\Q}{{\mathbb Q}}

\def\H{\mathbb H}

\newcommand{\bea}{\begin{eqnarray}} 
\newcommand{\eea}{\end{eqnarray}} 
\newcommand{\be}{\begin{equation}} 
\newcommand{\ee}{\end{equation}} 
\newcommand{\benn}{\begin{equation*}} 
\newcommand{\eenn}{\end{equation*}}

\def\({\left(}
\def\){\right)}

\def\k2{\frac{k}{2}}

\begin{document}
\date{\today}
\author{Kathrin Bringmann}
\address{Mathematical Institute\\University of
Cologne\\ Weyertal 86-90 \\ 50931 Cologne \\Germany}
\email{kbringma@math.uni-koeln.de}
\author{Larry Rolen}
\address{Mathematical Institute\\University of
Cologne\\ Weyertal 86-90 \\ 50931 Cologne \\Germany}
\email{lrolen@math.uni-koeln.de}
\thanks{The research of the first author was supported by the Alfried Krupp Prize for Young University Teachers of the Krupp foundation and the research leading to these results has received funding from the European Research Council under the European Union's Seventh Framework Programme (FP/2007-2013) / ERC Grant agreement n. 335220 - AQSER. The second author thanks the University of Cologne and the DFG for their generous support via the University of Cologne postdoc grant DFG Grant D-72133-G-403-151001011, funded under the Institutional Strategy of the University of Cologne within the German Excellence Initiative.}


\begin{abstract}
In analogy with the classical theory of Eichler integrals for integral weight modular forms, Lawrence and Zagier considered examples of Eichler integrals of certain half-integral weight modular forms. These served as early prototypes of a new type of object, which Zagier later called a quantum modular form. Since then, a number of others have studied similar examples. Here we develop the theory in a general context, giving rise to a well-defined class of quantum modular forms. Since elements of this class show up frequently in examples of combinatorial and number theoretical interest, we propose the study of the general properties of this space of quantum modular forms. We conclude by raising fundamental questions concerning this space of objects which merit further study.

\end{abstract}
\title[Eichler integrals and quantum modular forms]{Half-integral weight Eichler integrals and quantum modular forms}
\dedicatory{To Winnie Li, who has been a great inspiration, on the occasion of her birthday}
\maketitle

\section{Introduction and statement of results}

It is well-known that the derivative of a modular form is typically not a modular form. However, thanks to an identity of Bol \cite{Bol}, there exists a canonical differential operator $D^{k-1}\colon M_{2-k}^!\rightarrow M_k^!$ for $k\in\N$, where $D:=\frac1{2\pi i}\frac{\partial}{\partial\tau}$ and $M_{\ell}^!$ denotes the space of weight $\ell$ weakly holomorphic modular forms. Motivated by this, Eichler \cite{Eichler} considered the formal $(k-1)$-st antiderivative of a cusp form. Specifically, if $f(\tau)=\sum_{n\geq1}a_f(n)q^n$ (throughout $q:=e^{2\pi i \tau}$ with $\tau\in\H$) is a cusp form on $\operatorname{SL}_2(\Z)$, then we define the holomorphic Eichler integral by 
$\widetilde{f}(\tau):=\sum_{n\geq1}a_f(n) n^{1-k}q^n.$
Denoting by $|_{\ell}$ the Petersson slash operator in weight $\ell$ (defined in Section 2.1), we easily see from the remarks above that $D^{k-1}(\widetilde f\vert_{2-k}(1-S))=0$ (where $S:=(\begin{smallmatrix}0&-1\\ 1&0\end{smallmatrix}))$. Hence, $\widetilde f\vert_{2-k}(1-S)$ is a polynomial of degree at most $k-2$, called the period polynomial of $f$. Eichler integrals play a fundamental role in the study of integral weight modular forms. For example, as further elaborated on by Shimura \cite{Shimura} and Manin \cite{Manin}, the theory of Eichler integrals provides deep insights into the theory of elliptic curves and critical values of $L$-functions. For an interesting discussion of the general cohomology theory, see also \cite{Bruggeman}. 

Since half-integral weight modular forms (see Section 2.1 for the definition)  also encode deep arithmetic information, it is natural to ask what the analogous Eichler-Shimura theory is for half-integral weight. Here, the situation is more complicated. In particular, the operator $D^{k-1}$ no longer makes sense. In fact, it is a central problem in the theory of harmonic Maass forms to construct a suitable operator which plays a similar role as $D^{k-1}$ in half-integral weight (see, for example, \cite{FractionalDeriv}). The first pioneering examples of half-integral weight Eichler integrals were considered in connection with WRT invariants of 3-manifolds by Lawrence and Zagier \cite{Lawrence-Zagier}. Although the proof is more difficult for general weight modular forms, they formally considered the same Eichler integral as defined above for certain weight $3/2$ theta functions $\vartheta$. Even though it is impossible for $\widetilde{\vartheta}\vert_{\frac12}(1-\gamma)$ (where $\gamma$ is any element of the congruence subgroup of $\vartheta$) to be a polynomial of degree $k-2$ as in the integral weight case, they give a nice characterization of its modular properties as one approaches cusps. In fact, $\widetilde{\vartheta}$ can be extended to $\Q$, and the resulting function $\widetilde{\vartheta}\vert_{\frac12}(1-\gamma)$ becomes real-analytic on $\R\backslash\{\gamma^{-1}\infty\}$. This provides one of the first examples of the burgeoning new theory of quantum modular forms, laid out by Zagier \cite{ZagierQMF}, which we review in Section 2.2. Essentially, a quantum modular form of weight $k$ is a complex-valued function $f$ on $\Q$ whose modular obstructions, or cocycles, $f\vert_k(1-\gamma)$ are ``nicer'' than the original function in some analytic way. For example, $f$ is usually only well-defined on $\Q$, whereas $f\vert_k(1-\gamma)$ typically extends to an open set of $\R$ and is differentiable, smooth, etc.

Since \cite{Lawrence-Zagier}, there has been an explosion of research aimed at constructing examples of quantum modular forms related to non-holomorphic Eichler integrals, see for example \cite{Kimport,Bringmann-Creutzig-Rolen,Bringmann-Folsom-Rhoades,Folsom-Ono-Rhoades,ZagierVass}. For instance, quantum modular forms are closely tied to surprising identities relating the combinatorial generating functions counting ranks, cranks, and unimodal sequences \cite{Folsom-Ono-Rhoades}, and to the general theory of negative index Jacobi forms and Kac-Wakimoto characters \cite{Bringmann-Creutzig-Rolen}. In this paper, we elucidate the general picture in arbitrary half-integral weight. Although the previous proofs depended on the modular forms considered being theta functions, we show that a similar phenomenon is true in general, along the way constructing large families of quantum modular forms. Our main result is the following, where a more detailed definition of quantum modular forms and ``nice properties'' is given in Section 2.2. 
\begin{theorem}\label{mainthm}
If $f\in S_k(N)$ with $k\in\frac12+\N_0$ and $N\in4\N$, then $\widetilde f$ is a quantum modular form of weight $2-k$. In this instance, the ``nice'' property of the cocycle is that for every $\gamma\in\Gamma_0(N)$, $\widetilde f\vert_{2-k}(1-\gamma)$ is real-analytic except at $\gamma^{-1}\infty$.
\end{theorem}

\noindent
{\it Two remarks.}

\noindent 1) In fact, the proof of Proposition 2.1 shows that $\widetilde f$ satisfies a much stricter condition of quantum modularity and, in the language of Zagier, is a strong quantum modular form (cf. Section 2.2). 

\smallskip
\noindent
2)
The quantum modular forms in Theorem \ref{mainthm} were recently described in a different guise using the theory of mock theta functions (see Theorem 1.5 and Lemma 3.1 of \cite{RhoadesChoiLim}). However, we give a different proof of their quantum modularity, as the techniques here are of general interest. 
\smallskip

Although the definition of quantum modular forms by Zagier is (intentionally) vague, the evolution of the literature indicates that it is now worthwhile to split up the types of quantum modular forms which naturally arise into various categories. Theorem 1.1 shows that it makes sense to consider the vector space of Eichler integrals of cusp forms in $S_{k}(N)$ as an interesting space of quantum modular forms, and we propose this new area of study in a series of questions at the end of the paper.
In addition to the results of Theorem \ref{mainthm}, we give explicit formulas for the resulting quantum modular forms. Recall that for a half-integral weight cusp form $f(\tau)=\sum_{n\geq1} a_f(n)q^n$, its $L$-function is defined for $\operatorname{Re} (s)\gg0$ by 
$L_f(s):=\sum_{n\geq1}\frac{a_f(n)}{n^s}.$ 
More generally, consider the twisted $L$-function defined for $\operatorname{Re} (s)\gg0$ and $\frac dc\in\Q$ (throughout we assume that fractions are expressed in lowest terms) by
$L_f(\zeta_c^d;s):=\sum_{n\geq1}\frac{a_f(n)\zeta_c^{dn}}{n^s},$ 
where $\zeta_{a}^b:=e(\frac ab)$ with $e(x):=e^{2\pi ix}$, and define the function $Q_f\colon\Q\rightarrow\C$ by 
\[Q_f\left(\frac dc\right):=L_f\left(\zeta_c^d;k-1\right).\] 
We show that $Q_f$ is a quantum modular form.

\begin{corollary}
Assume the notation of Theorem \ref{mainthm}. Then $Q_f$ is a quantum modular form on $\Gamma_0(N)$, with $Q_f\vert_{2-k}(1-\gamma)$ real-analytic except at $\gamma^{-1}\infty$ for all $\gamma\in\Gamma_0(N)$.
\end{corollary}
\begin{remark}
We will see in Section 2.4 that $L_f(\zeta_c^d;s)$ has an analytic continuation to $\C$, so that $Q_f$ is well-defined.
\end{remark}

The paper is organized as follows. In Section 2, we review the definitions of half-integral weight modular forms and quantum modular forms, recall an auxiliary non-holomorphic Eichler integral considered in \cite{Lawrence-Zagier}, and reduce the statement of Theorem \ref{mainthm} to a certain claim about asymptotic expansions. We conclude Section 2 by giving several useful facts about $L$-functions needed for the proof of Theorem \ref{mainthm}. In Section 3, we complete the proof of Theorem \ref{mainthm} and Corollary 1.2. We conclude in Section 4 with a list of further questions raised by this paper.
\section*{Acknowledgements}
\noindent The authors are grateful to Robert C. Rhoades for enlightening comments and discussion, as well as the referee for useful suggestions which improved the paper.
\section{Preliminaries}
\subsection{Half-integral weight modular forms}
Here we review some standard definitions and facts concerning modular forms. Firstly, for $k\in\frac12\Z$, recall that the Petersson slash operator is defined for functions $f\colon\H\rightarrow\C$ and $\gamma=(\begin{smallmatrix}a&b\\ c&d\end{smallmatrix})\in\operatorname{SL}_2(\Z)$ (with $\gamma\in\Gamma_0(4)$ if $k\in\frac12+\Z$) by 
\[f\vert_{k}\gamma(\tau):=\begin{cases}(c\tau+d)^{-k}f\left(\frac{a\tau+b}{c\tau+d}\right)&\text{ for }k\in\Z,\\ \varepsilon_d^{2k}\left(\frac cd\right)(c\tau+d)^{-k}f\left(\frac{a\tau+b}{c\tau+d}\right)&\text{ for }k\in\frac12+\Z,\end{cases}\] 
where $(\frac{\cdot}{\cdot})$ denotes the Jacobi symbol and for odd $d$,
\[\varepsilon_d:=\begin{cases}1&\text{ if }d\equiv1\pmod4,\\ i&\text{ if }d\equiv 3\pmod4.\end{cases}\]
We require the following congruence subgroups of $\operatorname{SL}_2(\Z)$:
\begin{equation*}\begin{aligned}\Gamma_0(N)&:=\left\{\begin{pmatrix}a&b\\ c&d\end{pmatrix}\in\operatorname{SL}_2(\Z)\colon N|c\right\},\\
\Gamma_1(N)&:=\left\{\begin{pmatrix}a&b\\ c&d\end{pmatrix}\in\Gamma_0(N)\colon a\equiv d\equiv 1\pmod N\right\}.\end{aligned}\end{equation*}
We then make the following definition.
\begin{defn}
Let $k\in\frac12+\N_0$, $N\in4\N$, and $\chi$ be a Dirichlet character modulo $N$. Then a holomorphic function $f\colon\H\rightarrow\C$ is a cusp form of weight $k$ on $\Gamma_0(N)$ with Nebentypus $\chi$ if the following conditions hold:
\begin{enumerate}
\item For all $\gamma=(\begin{smallmatrix}a&b\\ c&d\end{smallmatrix})\in\Gamma_0(N)$, we have
$f\vert_{k}\gamma=\chi(d)f.$
\item As $\tau$ approaches any cusp of $\Gamma_0(N)$, $f$ decays exponentially fast.
\end{enumerate}

\end{defn}
We denote the space of cusp forms of weight $k$ on $\Gamma_0(N)$ and Nebentypus $\chi$ by $S_k(N,\chi)$. If $\chi$ is trivial, then we also use the notation $S_k(N)$. One may analogously define a space of cusp forms $S_k(\Gamma_1(N))$. We have the following well-known decomposition
\begin{equation}\label{decompgamma1}S_k(\Gamma_1(N))=\bigoplus_{\chi} S_k(N,\chi),\end{equation}
where $\chi$ runs over all even Dirichlet characters modulo $N$. 
Thus, in the study of modular forms on $\Gamma_1(N)$, it is often sufficient to consider modular forms on $\Gamma_0(N)$ with Nebentypus.

\subsection{Quantum modular forms}
In this subsection, we recall some definitions and examples of quantum modular forms.
Following Zagier, we make the following definition. 

\begin{defn}
A function $f\colon \mathbb{Q}\rightarrow\C$ is a quantum modular form of weight $k$ on a congruence subgroup $\Gamma$ if, for all $\gamma\in\Gamma$, the cocycle
$r_{\gamma}:=f|_k(1-\gamma)$
extends to an open subset of\, $\R$ and is analytically ``nice''. Here ``nice'' could mean continuous, smooth, real-analytic, etc. We say that $f$ is a strong quantum modular form if, in addition, $f$ has formal power series attached to each rational number which also have near-modularity properties (see \cite{ZagierQMF} for more details).

\end{defn}

\begin{remark} 

All of the quantum modular forms occurring in this paper have cocycles defined on $\R$ which are real-analytic except at one point. 

\end{remark}

One of the most striking examples of a quantum modular form is given by Kontsevich's ``strange function'' $F(q)$, as studied by Zagier in \cite{ZagierVass}, which is given by 

\begin{equation}\label{kontstrange}F(q):=\sum_{n\geq0}(q;q)_n,\end{equation}
where $(a;q)_n:=\prod_{j=0}^{n-1}(1-aq^j)$ denotes the usual $q$-Pochhammer symbol. This function is strange as it does not converge on any open subset of $\C$, but converges as a finite sum for $q$ any root of unity. Zagier's study of $F$ depends on the ``sum of tails'' identity 
\begin{equation}
\label{sumoftails}
\displaystyle\sum_{n\geq0}\left(\eta(\tau)-q^{\frac1{24}}\left(q;q\right)_n\right)=\eta(\tau)D\left(\tau\right)+\sqrt 6\widetilde{\eta}(\tau),
\end{equation}
\noindent
where $\eta(\tau):=q^{1/24}(q;q)_{\infty}$ and 
$D(\tau):=-\frac12+\sum_{n\geq1}\frac{q^n}{1-q^n}.$
The key observation of Zagier is that in (\ref{sumoftails}), the functions $\eta(\tau)$ and $\eta(\tau)D(\tau)$ vanish of infinite order as $\tau\rightarrow\frac hk$, so at a root of unity $\xi$, $F(\xi)$ is essentially the limiting value of the Eichler integral of $\eta$, which he showed has quantum modular properties \cite{ZagierVass}.

\subsection{Non-holomorphic Eichler integrals}

We now suppose that $f\in S_k(N)$ for $N\in4\N$ and $k\in\frac12+\N_0$. The main idea in studying the modularity properties of half-integral weight Eichler integrals, due to Lawrence and Zagier, is to introduce the non-holomorphic Eichler integral

\[f^*(\tau):=\frac{(-2\pi i)^{k-1}}{\Gamma(k-1)}\int_{\overline{\tau}}^{i\infty}f(w)(w-\tau)^{k-2}dw,\]
defined for $\tau\in\H^-:=\{u+iv\in\C\colon v<0\}.$ The point is that $\widetilde f$ has a $q$-series expansion (e.g., if $f$ is a theta function, then $\widetilde f$ is a partial theta function), while $f^*$ satisfies a nice transformation law. These transformation properties of $f^*$ transfer over to $\widetilde f$ near the real axis. For concreteness, we make the following definition. 

\begin{defn}
Let $f(\tau)$ and $g(\tau)$ be defined for $\tau\in\H$ and $\tau\in\H^-$, respectively. We say that the asymptotic expansions of $f$ and $g$ agree at a rational number $\frac dc$ if there exist $\beta(n)$ such that as $t\rightarrow0^+$,

\[f\left(\frac dc+\frac{it}{2\pi}\right)\sim\sum_{n\geq0}\beta(n)t^n,\quad\quad g\left(\frac dc-\frac{it}{2\pi}\right)\sim\sum_{n\geq0}\beta(n)(-t)^n.\]
\end{defn}
We now look at the transformation properties of $f^*$. We easily compute for 
 $\gamma=\(\begin{smallmatrix}a&b\\ c&d\end{smallmatrix}\)\in\Gamma_0(N)$ and $\tau\in\H^-$ that 

\begin{equation}\label{nonholtransform}f^*(\tau)-\chi_{-4}(d)f^*\vert_{2-k}\gamma(\tau)= \frac{(-2\pi i)^{k-1}}{\Gamma(k-1)}\int_{-\frac dc}^{i\infty}f(w)(w-\tau)^{k-2}dw,\end{equation}
where $\chi_{-4}$ is the Dirichlet character defined by $\chi_{-4}(n):=(\frac{-4}n)$.
This is the key transformation property giving rise to quantum modularity in Theorem \ref{mainthm}. The connection between $\widetilde f$ and $f^*$ is given by the following proposition, whose proof we defer to Section 3 (similar results were explored in less generality in e.g. \cite{Hikami,Lawrence-Zagier,ZagierVass}). 

\begin{proposition}\label{fooprop} Assuming the notation of Theorem \ref{mainthm}, the asymptotic expansions of $\widetilde f$ and $f^*$ agree at any $\frac dc\in\Q$. \end{proposition}

Thus, by Proposition \ref{fooprop}, $\widetilde f$ inherits the same transformation properties as $f^*$ as one approaches the real line. 
\subsection{Properties of $L$-functions}
In this subsection, we recall some basic properties of modular $L$-functions needed for the proof of Theorem \ref{mainthm}. 
Firstly, we require the following lemma. 
\begin{lemma}\label{vanishing}
Let $f\in S_k(N)$, where $k\in\frac12+\N_0$, $N\in4\N$, and $\frac dc\in\Q$. Then $L_f(\zeta_c^d;s)$ has an analytic continuation to $\C$ and $L_f(\zeta_c^d;-m)=0$ for $m\in\N_0$.

\end{lemma}

\begin{proof}
We first determine the modularity properties of $f_{\frac dc}(\tau):=\sum_{n\geq1}a_f(n)\zeta_c^{dn}q^n$. Note that
\[f_{\frac dc}(\tau)=\sum_{j=0}^{c-1}\zeta_c^{dj}\sum_{n\equiv j\pmod c}a_f(n)q^n.\]
It is well-known that
$\sum_{n\equiv j\pmod c}a_f(n)q^n\in S_k(\Gamma_1(N c^2)),$ 
and hence $f_{\frac dc}\in S_k(\Gamma_1(N c^2)).$ By the decomposition (\ref{decompgamma1}), we see that $f_{\frac dc}$ can be written as a finite sum $f_{\frac dc}=\sum_{j=0}^{N_0}f_{\frac dc,j}$ where $f_{\frac dc,j}\in S_k(Nc^2,\chi_j)$ and $\chi_j$ is a Dirichlet character modulo $Nc^2$. Clearly $L_f(\zeta_c^d;s)=L_{f_{\frac dc}}(s)$, so it suffices to show that the lemma holds for $L_g(s)$ where $g\in S_{k}(M,\chi)$ with $M\in4\N$ and $\chi$ is a Dirichlet character modulo $M$. For such a $g$, the proof of the analytic continuation and vanishing condition of $L_g$ is essentially classical, due to Hecke (see Theorem 14.7 of \cite{Iwaniec-Kowalski} and Satz 4 of \cite{Petersson}). However, since the multiplier is different in half-integral weight, for completeness we prove it directly.

For this, we first recall the action of the Fricke involution, given by 
\begin{equation}\label{WNDefn}g\vert_kW_N(\tau):=\left(-i\sqrt N\tau\right)^{-k}g\left(-\frac1{N\tau}\right).\end{equation} It is well-known that $g\vert_kW_N\in S_k(M,\chi(\frac{N}{\cdot}))$ (see Section 3 of \cite{Bruinier}). We now consider the completed $L$-function 
\[\Lambda_g(s):=\int_0^{\infty}g\left(\frac{iv}{\sqrt N}\right)v^{s-1}dv.\] By a simple calculation, one sees that the completed $L$-function factors as 
\begin{equation}\label{completedsplitting}\Lambda_g(s)=\left(\frac{\sqrt{N}}{2\pi}\right)^s\Gamma(s)L_g(s).\end{equation} The key property of $\Lambda_g(s)$ is its functional equation. To determine it, we begin by splitting $\Lambda_g(s)$ into two pieces as
\[\Lambda_g(s)=\int_0^{1}g\left(\frac{iv}{\sqrt N}\right)v^{s-1}dv+\int_1^{\infty}g\left(\frac{iv}{\sqrt N}\right)v^{s-1}dv,\]
and make a change of variables in the first integral to obtain
\[\Lambda_g(s)=\int_1^{\infty}g\left(\frac{i}{v\sqrt N}\right)v^{-s-1}dv+\int_1^{\infty}g\left(\frac{iv}{\sqrt N}\right)v^{s-1}dv.\]
Inserting (\ref{WNDefn}), we find
\begin{equation}\label{completedint}\Lambda_g(s)=\int_1^{\infty}g\vert_kW_N\left(\frac{iv}{\sqrt N}\right)v^{k-s-1}dv+\int_1^{\infty}g\left(\frac{iv}{\sqrt N}\right)v^{s-1}dv.\end{equation}
Since both $g$ and $g\vert_kW_N$ are cusp forms, and hence have rapid decay as $v\rightarrow\infty$, (\ref{completedint}) immediately gives an analytic continuation of $\Lambda_g(s)$ to $\C$.
The analytic continuation of $L_g(s)$ now follows immediately from (\ref{completedsplitting}) and the fact that $1/\Gamma(s)$ has no poles. 
The integral representation (\ref{completedint}) also directly implies the functional equation, namely
\begin{equation}\label{functionaleqn}\Lambda_g(s)=\Lambda_{g\vert_kW_N}(k-s).\end{equation}
Now the claims follow, as the Gamma factor $\Gamma(s)$ in (\ref{completedsplitting}) forces $L_g(s)$ to have zeros at non-positive integers, as $\Gamma(s)$ has a pole at these locations whereas the right hand side of (\ref{functionaleqn}) does not.
\end{proof}
Besides the analytic continuation of our $L$-functions, we require a growth estimate as $\vert\operatorname{Im}(s)\vert\rightarrow\infty$. For our purposes, the following basic lemma suffices. Since the proof is a standard application of the functional equation and the Phragm\'en-Lindel\"of principle, we omit the proof. 
\begin{lemma}
\label{Lfunctiongrowth}
For fixed $x\in\R$, $L_{f}(\zeta_c^d;x+it)$ grows at most polynomially in $t$ as $|t|\rightarrow\infty$.
\end{lemma}

\section{Proof of Theorem \ref{mainthm} and Corollary 1.2}

In this section, we show Proposition \ref{fooprop}, use it to prove Theorem \ref{mainthm}, and then deduce Corollary 1.2. 
To prove Proposition 2.1, we compute the asymptotic expansions of $\widetilde f$ and $f^*$ separately to show that they agree. This is done using the Mellin transform, defined by $\mathcal M(f)(s):=\int_0^{\infty}f(t)t^{s-1}dt$, which provides a fundamental connection between the asymptotic expansion of one function and the poles of another. Specifically, we require the following result for computing asymptotic expansions, which is a special case of Theorem 4 (i) of \cite{Flajolet}. 
\begin{lemma}
\label{AsympExpExists}
Let $F(x)$ be continuous on $(0,\infty)$ with Mellin transform $\mathcal M(F)(s)$ converging on a right half-plane $\operatorname{Re}(s)>\alpha$. Assume that $\mathcal M(F)(s)$ can be meromorphically continued to the half-plane $\operatorname{Re}(s)>\beta$,
where $\beta<\alpha$, with a finite number of poles $a_0,a_1,\ldots, a_M$ in that half-plane, each simple with residue $\alpha_j$. Further assume that $\mathcal M(F)(s)$ is analytic on the line $\operatorname{Re}(s)=\beta$ and that in the right half-plane $\operatorname{Re}(s)>\beta$, the estimate
\[\mathcal M(F)(s)=O\left(|s|^{-r}\right)\]
holds as $|s|\rightarrow\infty$ for some $r>1$. Then we have the asymptotic expansion
\[F(x)\sim\sum_{j=0}^M\alpha_jx^{-a_j}+O\left(x^{-\beta}\right).\]

\end{lemma}
For later use, we also recall that for a general function $f$ with $f(x)=O(x^{\alpha}) \text{ as }x\rightarrow0^+$ and $f(x)=O(x^{\beta})\text{ as }x\rightarrow\infty,$
$\mathcal M(f)(s)$ converges in the strip $-\alpha\leq\operatorname{Re}(s)\leq-\beta$.

In order to apply Lemma \ref{AsympExpExists} to the function $\widetilde f(\frac dc+\frac{it}{2\pi})$, we first compute its Mellin transform by integrating termwise

\begin{equation}\label{neednow}\begin{aligned}\mathcal{M}\left(\widetilde f\left(\frac dc+\frac{it}{2\pi}\right)\right)\left(s\right)&=\int_0^{\infty}\sum_{n\geq1}n^{1-k}\zeta_c^{nd}a_f(n)e^{-nt}t^{s-1}dt &=\Gamma(s)L_f\left(\zeta_c^d;k-1+s\right).\end{aligned}\end{equation}
We next show that this Mellin transform satisfies the conditions of Lemma \ref{AsympExpExists}. Firstly, note, using the right-hand side of (\ref{neednow}), that $\mathcal{M}(\widetilde f(\frac dc+\frac{it}{2\pi}))(s)$ is convergent on some right half-plane, and by Lemma \ref{vanishing}, it has an analytic continuation to $\C$ with poles only at non-positive integers (coming from the Gamma factor). 
To estimate the growth of the Mellin transform in vertical strips, we first recall Stirling's estimate (see 5.11.9 of \cite{Nist}):
\begin{equation}\label{Stirling}\left\vert\Gamma(x+iy)\right\vert=\sqrt{2\pi}|y|^{x-\frac12}e^{-\frac{\pi|y|}2}\ \text{  as }|y|\rightarrow\infty.\end{equation}
By Lemma 2.3 and (\ref{Stirling}), the Mellin transform is thus of rapid decay for fixed $\operatorname{Re}(s)$ as $\vert\operatorname{Im}(s)\vert\rightarrow\infty$. 
Hence, letting $\beta\rightarrow -\infty$ in Lemma \ref{AsympExpExists}, we directly find
\begin{equation}\label{HolEichAsympFinal}\widetilde f\left(\frac dc+\frac{it}{2\pi}\right)\sim\sum_{n\geq0}\frac{(-1)^n}{n!}L\left(\zeta_c^d;k-1-n\right)t^n\text{ as }t\rightarrow0^+,\end{equation}

We now turn to computing the asymptotic expansion of $f^*(\frac dc-\frac{it}{2\pi})$ as $t\rightarrow0^+$. We begin by expanding $f^*$.
By a change of variables, we have for $n>0$ and $\tau=u+iv$ with $v<0$

\[\int_{\overline{\tau}}^{i \infty}e^{2\pi i nw}(w-\tau)^{k-2}dw=i^{k-1}(2\pi n)^{1-k}e^{2\pi i n \tau}\Gamma(k-1,4\pi n |v|),\]
where $\Gamma(s,x):=\int_x^{\infty}t^{s-1}e^{-t}dt$ denotes the incomplete Gamma function. Thus, integrating term-by-term, we obtain

\begin{equation}\label{fstarexp}f^*(\tau)=\frac1{\Gamma(k-1)}\sum_{n\geq1}a_f(n) n^{1-k}e^{2\pi i n\tau}\Gamma(k-1,4\pi n|v|).\end{equation}
Taking the Mellin transform of the right hand side of (\ref{fstarexp}), we obtain

\begin{equation}\label{mellinnonholeich}\begin{aligned}\mathcal{M}\left(f^*\left(\frac dc-\frac{it}{2\pi}\right)\right)(s)&=\frac1{\Gamma(k-1)}L_f\left(\zeta_c^d;k-1+s\right)\mathcal{M}\left(e^t\Gamma(k-1,2t)\right)(s).\end{aligned}\end{equation}

In order to use Lemma \ref{AsympExpExists} to compute the asymptotic expansion of $f^*$, we first determine the location and residues of the poles of $\mathcal{M}(f^*(\frac dc-\frac{it}{2\pi}))(s)$ using the representation on the right-hand side of (\ref{mellinnonholeich}).

\begin{lemma}\label{crap}
Assuming the notation above, the following are true:
\begin{enumerate}
\item The function $\mathcal{M}(f^*(\frac dc-\frac{it}{2\pi}))(s)$ has a simple pole at $s=-n \in-\N_0$ with residue 
\[\frac{ L_f\left(\zeta_c^d;k-1-n\right)}{n!}.\]
\item The function $\mathcal{M}(f^*(\frac dc-\frac{it}{2\pi}))(s)$ is holomorphic for $s\not\in-\N_0$. 
\end{enumerate}
\end{lemma}
\begin{proof}
(1) We begin by rewriting the Mellin transform $\mathcal M(e^t\Gamma(k-1,2t))(s)$ in a more convenient form. 
Namely, we claim that for $\text{Re}(s)>0$ and $\text{Re}(k+s-1)>0$, we have
\begin{equation}\label{bleh}\mathcal M\left(e^t\Gamma\left(k-1,2t\right)\right)(s)=\beta\left(\frac12;s,2-k-s\right)\Gamma\left(k-1+s\right),\end{equation}
where, for $\operatorname{Re}(a)>0$, $\operatorname{Re}(b)>0$, and $0\leq z\leq 1$, 
\begin{equation*}\beta(z;a,b):=\int_0^zt^{a-1}(1-t)^{b-1}\mathrm dt\end{equation*}
is the incomplete beta function. To see (\ref{bleh}), define for $|z|<1$ the Gaussian hypergeometric series 
\begin{equation}\label{2F1Defn} _2F_1\left(a,b;c;z\right):=\sum_{n\geq0}\frac{(a)_n(b)_nz^n}{(c)_nn!},\end{equation}
where $(x)_n:=\prod_{j=0}^{n-1}(x+j)$ is the Pochhammer symbol (note that $ _2F_1$ is sometimes denoted by $F$ in \cite{Nist}). Now use (8.14.6) of \cite{Nist}, which states that for $\text{Re}(r)>-1,\text{ Re}(a+b)>0,$ and $\text{ Re}(a)>0,$
\[\int_0^{\infty}x^{a-1}e^{-rx}\Gamma(b,x) dx=\frac{\Gamma(a+b)}{a(1+r)^{a+b}} \text{  }_2F_1\left(1,a+b;1+a;\frac r{1+r}\right).\]
Choosing $b=k-1$, $a=s$, and $r=-\frac12$, we obtain, after a change of variables,
\begin{equation}\label{labelfoo}\mathcal M\left(e^t\Gamma(k-1,2t)\right)(s)=2^{k-1}\frac{\Gamma(s+k-1)}{s}\text{}_2F_1(1,s+k-1;1+s;-1).\end{equation}
Here we remark that the specialization of $ _2F_1$ in (\ref{labelfoo}) is defined by analytic continuation, as discussed in Section 15.2 of \cite{Nist}.  We then use (8.17.9) of \cite{Nist} which states that 
\begin{equation}\label{beta2f1}\beta(x;a,b)=\frac{x^a(1-x)^{b-1}}{a}\text{}_2F_1\left(1,1-b;a+1;\frac{x}{x-1}\right).\end{equation}
Plugging in $x=\frac12$, $a=s$, and $b=2-k-s$ and using (\ref{labelfoo}) gives (\ref{bleh}). 

To determine the poles and residues of the incomplete beta function, we make use of the beta function $\beta(a,b):=\beta(1;a,b)$ and the regularized beta function $I(z;a,b):=\frac{\beta(z;a,b)}{\beta(a,b)}.$ Using the standard fact that \begin{equation}\label{betafactor}\beta(a,b)=\frac{\Gamma(a)\Gamma(b)}{\Gamma(a+b)},\end{equation}
along with the Euler reflection formula for $\Gamma(s)$, we obtain 

\begin{equation}\label{mellinrewrite}\begin{aligned} \beta\left(\frac12;s,2-k-s\right)\Gamma\left(k-1+s\right)=
 \Gamma(s)\frac{\pi I\left(\frac12;s,2-k-s\right)}{\Gamma(2-k)\sin\left(\pi(k-1+s)\right)}.\end{aligned}\end{equation}
Combining (\ref{mellinnonholeich}), (\ref{bleh}), and (\ref{mellinrewrite}) and again using Euler's reflection formula, we find that 
\begin{equation*}\mathcal{M}\left(f^*\left(\frac dc-\frac{it}{2\pi}\right)\right)(s)= L_f\left(\zeta_c^d;k-1+s\right)  \Gamma(s)\frac{(-1)^{k+\frac12} I\left(\frac12;s,2-k-s\right)}{\sin\left(\pi(k-1+s)\right)}.\end{equation*}
We next specialize to $s=-n$. Note that $\sin\left(\pi(k-1-n)\right)=(-1)^{k+\frac12+n}$. Using Lemma 2.2 again, Lemma 3.1 (1) thus follows immediately once we have established that
\begin{equation}\label{foobar}I\left(\frac12;-n,2-k+n\right)=1.\end{equation}

To see (\ref{foobar}), first note the following two identities, stated in (26.5.2) and (26.5.15) of \cite{AS}, respectively:
\begin{equation}\label{foo1}I(x;a,b)=1-I(1-x;b,a),\end{equation}
\begin{equation} \label{recurrencebetareg} I(x;a,b)=\frac{\Gamma(a+b)}{\Gamma(a+1)\Gamma(b)}x^a(1-x)^{b-1}+I(x;a+1,b-1).\end{equation}
For $n\in\N_0$, (\ref{recurrencebetareg}) implies that 
\begin{equation}\label{anotherrecurr}I\(\frac12;2-k+n,-n\)=I\(\frac12;2-k,0\),\end{equation}
 since the factor $\frac{\Gamma(2-k)}{\Gamma(2-k+1)\Gamma(n)}$
 is zero. Using (\ref{recurrencebetareg}) again implies that 
\begin{equation}\label{lastrecurr}I\(\frac12;2-k,0\)=I\(\frac12;1-k,1\)-2^{k-1}.\end{equation}
The right-hand side of (\ref{lastrecurr}) may now be evaluated directly. For this, we note that 
$ _2F_1(1,0;2-k;-1)=1$, as all but the first term in (\ref{2F1Defn}) vanish, and hence, by (\ref{beta2f1}), $\beta(\frac12;1-k,1)=\frac{2^{k-1}}{1-k}.$ Although this specialization is on the border of the region of convergence $|z|<1$ of $ _2F_1$, by Abel's Lemma the value of the analytic continuation to a point on the boundary of convergence is the value at that point, assuming it exists. Now note that by (\ref{betafactor}), $\beta(1-k,1)=(1-k)^{-1}$, and hence by the definition of $I$, we have that $I(\frac12;1-k,1)=2^{k-1}$. Thus, by (\ref{anotherrecurr}) and (\ref{lastrecurr}), we have $I(\frac12;2-k+n,-n)=0$.
But then (\ref{foo1}) gives

\begin{equation*}I\left(\frac12;-n,2-k+n\right)=1-I\left(\frac12;2-k+n,-n\right) =1,\end{equation*}
proving (\ref{foobar}). 

\noindent
(2) Before proving that $\mathcal{M}(f^*(\frac dc-\frac{it}{2\pi}))(s)$ is holomorphic for $s\not\in-\N_0$, we first analyze where potential poles could arise. To do this, note that using 
(\ref{mellinnonholeich}), (\ref{labelfoo}), and the
standard Pfaff transformation formula
\begin{equation}\label{Pfaff} _2F_1(a,b,c;z)=(1-z)^{-b}\, _2F_1\(b,c-a;c;\frac z{z-1}\)\end{equation}
 together gives
\begin{equation}\label{simpler}\mathcal{M}\(f^*\(\frac dc-\frac{it}{2\pi}\)\)(s)=\frac{1}{2^s\Gamma(k-1)}L_f\(\zeta_c^d;k-1+s\)\Gamma(s+k-1)\frac{ _2F_1\(s+k-1,s;1+s;\frac12\)}s.\end{equation}
Hence, using the fact that the only possible poles of $ _2F_1(a,b;c;z)$ for $|z|<1$ occur if $c\in-\N_0$, together with Lemma 2.2, we see that $\mathcal M(f^*(\frac dc-\frac{it}{2\pi}))(s)$ can only possibly have poles for $s\in-\N_0$ or for $s\in\frac12+\Z$ with $s\leq 1-k$. 

Thus, to complete the proof of (2)  we just need to prove that there are no poles for $s\in\frac12+\Z$ with $s\leq 1-k$. We thus fix $s$ to be such a half-integer. Now it suffices by (\ref{simpler}) to show that 
$L_f\(\zeta_c^d;k-1+s\)=0$, as the only potential pole arises from a simple pole in the function $\Gamma(s+k-1)$.
However, by the assumption on $s$, we note that $k-1+s\in-\N_0$, so that the desired vanishing condition is stated in Lemma 2.2.
\end{proof}
We are now ready to provide the asymptotic expansion of $f^*$, which, by comparison with (\ref{HolEichAsympFinal}), proves Proposition 2.1.
One may easily show, using (\ref{beta2f1}), (\ref{Pfaff}), and (\ref{Stirling}), that $\beta(\frac12;s,2-k-s)$ is bounded in vertical strips for $\vert\operatorname{Im}(s)\vert$ large enough.
By this fact, together with (\ref{mellinnonholeich}), (\ref{bleh}), Lemma \ref{Lfunctiongrowth}, and (\ref{Stirling}), we find that $\mathcal{M}(f^*(\frac dc-\frac{it}{2\pi}))(s)$ is of rapid decay in vertical strips as $\vert\operatorname{Im}(s)\vert\rightarrow\infty$.
Hence, we can directly plug Lemma \ref{crap} into Lemma \ref{AsympExpExists} to find that
\begin{equation}\label{nonholfinalexp} f^*\left(\frac dc-\frac{it}{2\pi}\right)\sim\sum_{n\geq0}\frac{1}{n!}L\left(\zeta_c^d;k-1-n\right)t^n\text{ as }t\rightarrow0^+,\end{equation}
which establishes Proposition 2.1.

We now have all of the pieces needed to prove Theorem \ref{mainthm}.
\begin{proof}[Proof of Theorem \ref{mainthm}]
By (\ref{nonholfinalexp}), we find that $f^*(\frac dc-\frac{it}{2\pi})$ has a well-defined limit as $t\rightarrow0^+$, namely, 
$\lim_{t\rightarrow0^+}f^*(\frac dc-\frac{it}{2\pi})=L(\zeta_c^d;k-1).$
By (\ref{nonholtransform}), we see that $f^*$ is a quantum modular form, as for $x\in\R$, the cocycle 
$\int_{-\frac dc}^{i\infty}f(w)(w-x)^{k-2}dw$ 
is real-analytic on $\R\setminus\{-\frac dc\}$. By  (\ref{HolEichAsympFinal}), we see that $\widetilde f$ also extends to a function on $\Q$ with the same values as $f^*$, and hence is a quantum modular form. The claim that it is a strong quantum modular form follows from Proposition \ref{fooprop}. 
\end{proof}
The proof of Corollary 1.2 is now straightforward.
\begin{proof}[Proof of Corollary 1.2]
As in the proof of Theorem \ref{mainthm}, we see by (\ref{HolEichAsympFinal}) that  $\widetilde f$ is a quantum modular form, and that its values at the rational point $\frac dc$ is $L_f(\zeta_c^d;k-1)$. 
\end{proof}

\section{Questions and outlook}
Theorem 1.1 and Corollary 1.2 provide a canonical family of quantum modular forms, which gives a different perspective on the quantum modular forms defined in \cite{RhoadesChoiLim}. As discussed in Sections 1 and 2.2, this family includes many examples in the literature, such as Kontsevich's strange function (\ref{kontstrange}) \cite{ZagierVass}, quantum modular forms arising from specializations of the crank and rank generating functions \cite{Folsom-Ono-Rhoades}, and decomposition formulas for Kac-Wakimoto characters \cite{Bringmann-Creutzig-Rolen}. Due to the number of connections with combinatorics and number theory, we propose that studying these quantum modular forms in more detail is worthwhile. Thus, we make the following definition.
\begin{defn}
We say that a quantum modular form is an Eichler quantum modular form if its values on $\Q$ are equal to the radial limits from inside the unit disk of $\widetilde f$ for some cusp form $f$. 
\end{defn}
The results presented here leave many open questions about the structure of these quantum modular forms. We conclude by giving a few natural questions that merit further investigation.
\begin{questions}
\begin{enumerate}
\item If we replace the cusp form $f$ by a holomorphic modular form which is not cuspidal, then an inspection of the proof of Theorem \ref{mainthm} shows that the same calculations formally hold, but the Eichler integral has a mild singularity at cusps at which $f$ is not cuspidal. Is there a natural (and nontrivial) way to subtract a canonical holomorphic function which gives an associated quantum modular form? 
\item If $f$ is a unary theta function, then $\widetilde f$ is a partial theta function. Can one add a non-holomorphic completion term to the partial theta function which corrects the modularity transformations on the upper half plane as well as on $\Q$? More generally, for any cusp form $f$, are the modularity transformations proven in Theorem \ref{mainthm} a consequence of a more general non-holomorphic completion on the upper or lower half plane?
\item What are the arithmetic properties of Eichler quantum modular forms? For example, inspired by the congruences proven in \cite{Andrews}, is there a general theory of congruences for the coefficients of $\widetilde{f}(1-q)$ (note the slight abuse of notation here) or for the coefficients of the asymptotic expansion of $\widetilde f$ as $\tau$ approaches a given root of unity? In particular, if $f$ is a theta function, can one show in a uniform manner whether or not there are infinitely many such congruences, or does there exist a ``Sturm-type'' theorem which gives a finite condition to verify congruences?
\item What does the Shimura correspondence tell us about the structure of half-integral weight Eichler quantum modular forms, given that the theory of integral weight Eichler integrals is much simpler? 

\end{enumerate}
\end{questions}

\end{document}